\documentclass[10pt]{article}
\usepackage{geometry}                
\usepackage{graphicx}
\usepackage{amssymb,amsthm,amsmath}
\usepackage{epstopdf}
\usepackage{pdfsync}

\usepackage{color}

\title{Homogenization of quadratic convolution energies\\
in periodically perforated domains}
\author{
{\sc Andrea Braides
}
\\ Dipartimento di Matematica,
 Universit\`a di Roma `Tor Vergata'\\
via della Ricerca Scientifica, 00133 Rome, Italy\\
\and
{\sc Andrey Piatnitski}
\\
The Arctic University of Norway, UiT,  Campus
Narvik,\\ P.O. Box 385, Narvik 8505, Norway \\ and \\ Institute for Information Transmission Problems
of RAS,\\ 127051 Moscow, Russia 
\\
}
\date{
}                                      

\newtheorem{definition}{Definition}[section]
\newtheorem{lemma}[definition]{Lemma}
\newtheorem{theorem}[definition]{Theorem}
\newtheorem{proposition}[definition]{Proposition}

\newtheorem{remark}[definition]{Remark}
\newtheorem{example}[definition]{Example}

\def\e{\varepsilon}
\def\dx{\,dx}

\def\ZZ{\mathbb{Z}}
\def\NN{\mathbb{N}}
\def\rr{\mathbb{R}}

\def \trait (#1) (#2) (#3){\vrule width #1pt height #2pt depth #3pt}
\def \qed{\hfill
        \trait (0.1) (6) (0)
        \trait (6) (0.1) (0)
        \kern-6pt
        \trait (6) (6) (-5.9)
        \trait (0.1) (6) (0)
\medskip}

\begin{document}
\maketitle

\noindent
{\bf Abstract.}
We prove a homogenization theorem for quadratic convolution energies defined in perforated domains.
The corresponding limit is a Dirichlet-type quadratic energy, whose integrand is defined by a non-local cell-problem formula.
The proof relies on an extension theorem from perforated domains belonging to a wide class containing compact periodic perforations.

\bigskip


\section{Introduction}
In this paper we perform a limit analysis for a class of convolution functionals that may be interpreted as describing macroscopic features of  biological systems. Indeed, the study of macroscopic properties for complex biological systems and models of population dynamics can be reduced to studying the evolution of the so-called one-point correlation function describing the population density in the system. An important feature of the corresponding equation is that it is nonlocal with respect to spatial variables, and that the nonlocal operator is of convolution type, see \cite{KKP2008, FKK2012} for further details.
On a fixed domain $\Omega$ the energy of such a system in the stationary regime is given by
\begin{eqnarray}\label{def-Fe-solid}
{1\over \e^{d+2}}\int_{\Omega\times\Omega} a\Bigl({y-x\over\e}\Bigr)
(u(y)-u(x))^2dy \dx,
\end{eqnarray}
where $a$ is a convolution kernel with good integrability properties.

We note that if $u\in C^1(\Omega)$ then
$u(y)-u(x)\approx\langle\nabla u(x),y-x\rangle$ and, using the change of variables $y=x+\e\xi$,
\begin{eqnarray}\label{def-Fe-solid-2}
{1\over \e^{d+2}}\int_{\Omega\times\Omega} a\Bigl({y-x\over\e}\Bigr)
(\langle\nabla u(x),y-x\rangle)^2dy \dx\longrightarrow \int_{\Omega} \int_{\rr^d}a(\xi)
(\langle\nabla u(x),\xi\rangle)^2d\xi \dx,
\end{eqnarray}
as $\e\to 0$, so that the quadratic functional
 \begin{eqnarray}\label{def-Fe-solid-3}
 \int_\Omega \langle A\nabla u,\nabla u\rangle \dx, \qquad\hbox{ with }\qquad
 \langle A z, z\rangle= \int_{\rr^d}a(\xi)
(\langle z,\xi\rangle)^2d\xi,
\end{eqnarray}
gives an approximation of \eqref{def-Fe-solid}, or, conversely, \eqref{def-Fe-solid} gives a more general form of quadratic energies allowing for interactions between points at scale $\e$. In terms of $\Gamma$-convergence this computation can be reworked as a $\Gamma$-limit and the corresponding convergence of minumum problems. The convergence of the corresponding operators, also with periodic perturbations of the convolution kernel, has been studied in \cite{PZ}.

In the case of inhomogeneous media with a periodic microstructure and with zones where we do not have interactions, the model may be set in a perforated domain, where the energies are integrated only on the complement of a `perforation'. In the corresponding equations we will obtain a homogeneous (non-local) Neumann boundary condition on the perforation.
A general periodically perforated domain is obtained by intersecting $\Omega$ with a periodic connected Lipschitz set $E_\delta=\delta E$ with small period $\delta$. In the case of energies of convolution type the relevant scale of the period $\delta$ is of order $\e$. Indeed, otherwise, if $\e<\!<\delta$ then we may apply a separation of scales argument and reduce to the classical problem of the asymptotic description of perforated domain for the quadratic form \eqref{def-Fe-solid-3}, while if $\delta<\!<\e$ the effect of the perforation is averaged out giving as a limit the energy in \eqref{def-Fe-solid-3} multiplied by the constant $|E\cap(0,1)^d|$ (i.e., the constant density of the weak limit of $\chi_{\delta E}$). We will then directly consider energies of the form
\begin{eqnarray}\label{def-Fe-perforated}
F_\e(u)={1\over \e^{d+2}}\int_{(\Omega\cap E_\e)\times(\Omega\cap E_\e)} a\Bigl({y-x\over\e}\Bigr)
(u(y)-u(x))^2dy \dx
\end{eqnarray}
(i.e., with $\delta=\e$). In order to avoid technicalities, we will restrict to the simplified case of `compact perforations'; i.e., when
\begin{eqnarray}\label{def-perforation}
E=\rr^d\setminus (K_0+\ZZ^d),
\end{eqnarray}
where $K_0$ is a Lipschitz compact set such that $(K_0+j)\cap (K_0+j')=\emptyset$ if $j,j'\in \ZZ^d$ and $j\neq j'$.

In the asymptotic analysis of usual (quadratic) integral functionals on perforated domains one considers energies
\begin{eqnarray}\label{def-Fe-perforated-solid}
E_\e(u)=\int_{\Omega\cap E_\e}\langle A\nabla u,\nabla u\rangle \dx
\end{eqnarray}
or the corresponding elliptic differential operators with Neumann conditions on the perforation boundary.
Early homogenization results for such operators were obtained in \cite{MK,K,CS,V}, from which many other works on the topic followed.
The key argument in the asymptotic analysis both of energies and operators is an extension lemma, which allows to extend functions in $H^1(\Omega\cap E_\e)$ to functions in $H^1(\Omega')$ for $\Omega'\subset\subset \Omega$ controlling the $H^1$ and $L^2$ norms of the extension with those of the original function in the perforated domain independently of $\e$ (see \cite{ACDP,BDF}). In this way we may define an $L^2_{\rm loc}$-limit of sequences of functions with bounded energy as the limit of their extensions. Our first result is a similar Extension Theorem, which can be stated as follows:  for a fixed $r_0>0$ there exist $r>0$ and a constant $C$ such that for every fixed $\Omega'$ and for each function $u$ with bounded energy there exists an extension $v$ of function $u$ that  satisfies for all sufficiently small $\varepsilon$ the inequality
\begin{equation}\label{dva-intro}
\int_{(\Omega'\times\Omega')\cap\{|x-y|\le \e r\}}
(v(y)-v(x))^2 dy \dx\le C \int_{(\Omega\cap E_\e)\times(\Omega\cap E_\e)\cap\{|y-x|<\e r_0\}}(u(y)-u(x))^2dy \dx\,.
\end{equation}
The construction of the extended function $v$ can be achieved by a `reflection' argument close to $\partial E_\e$. Note however that the estimate of the energy is a little trickier than in the usual `local' case, since for the extended function we will have interactions between the function inside and outside the perforation.

Under the assumption that
\begin{eqnarray}\label{bou_a}
\hbox{there exist a constant $c>0$ and $r_0>0$ such that
 }a(z)\ge c \hbox{ if }|z|\le r_0
\end{eqnarray}
the term on the right-hand side of \eqref{dva-intro} can be estimated in terms of $F_\e(u)$.
Note that if hypothesis \eqref{bou_a} fails, the set of points $(x,y)\in E\times E$ with $a(y-x)\neq 0$ may be composed of disconnected components and as a result no controlled extension be possible.

Since, unlike for the usual Dirichlet integrals, inequality \eqref{dva-intro} does not imply the boundedness of the extensions in $H^1$, in order to conclude the argument, the extension theorem must be coupled with a Compactness Result for non-perforated energies (see \cite{2019BP}), now applied to the extended functions, which allows to conclude that from each sequence with equi-bounded energy we can extract a subsequence such that the corresponding extensions converge in $L^2_{\rm loc}$ to some $u\in H^1(\Omega)$.

Once this result is obtained, we can compute the $\Gamma$-limit of $F_\e$ with respect to the convergence described above.
The limit is a `classical' local quadratic energy
\begin{equation}\label{homfuneq-intro}
F_{\rm hom}(u)=\int_\Omega \langle A_{\rm hom}\nabla u,\nabla u\rangle\dx,
\end{equation}
with domain $H^1(\Omega)$, where $A_{\rm hom}$ is a symmetrix matrix
given by the cell-problem formula
\begin{equation}\label{cpformula-intro}
 \langle A_{\rm hom}z,z\rangle=\inf\Bigl\{\int_{(0,1)^d\cap E}\int_{E} a(\xi)(w(y+\xi)-w(y))^2 d\xi\,dy:
 w(y)-\langle z,y\rangle\hbox{ is $1$-periodic}\Bigr\}.
\end{equation}

It must be noted that in formula \eqref{cpformula-intro} the inner integral is performed on the whole $E$, highlighting that long-range interactions cannot be neglected. The treatment of these long-range interactions is the source of most of the technical points in the proof of the homogenization results. In order to control them we have to make some assumption on the decay of $a$, which we state as
 \begin{eqnarray}\label{ass-aa}
0 \le a(\xi)\le C {1\over (1+|\xi|)^{d+2+\kappa}}.
\end{eqnarray}
for some $\kappa>0$. Note that this is slightly more restrictive that assuming the finiteness of second moments of $a$, which is a necessary condition by \eqref{def-Fe-solid-3}.

This homogenization theorem is achieved by  first using the blow-up method of Fonseca and M\"uller \cite{FM,BMS}, and then reducing to one-periodic interactions by a convexity argument. The use of formula \eqref{cpformula-intro} allows the construction of recovery sequence for which we can control long-range interactions. It is worth remarking that the use of the blow-up method is possible thanks to the `vanishing non-locality' of the energies. We note the analogy with the results on `perforated' discrete domains (i.e., defined on scaled periodic lattices with some missing sites) studied in \cite{2015BCP} Section 3.

\section{Setting of the problem}

We consider a convolution kernel $a:\rr^d\to \rr^+$ with $ \int_{\rr^d}a(\xi)\,d\xi>0$ and such that \eqref{ass-aa} holds
for some $\kappa>0$.

We suppose that

\def\Kz{K_0}
$\bullet$  $\Kz$ is a compact subset of $\rr^d$ with Lipschitz boundary
such that $(\Kz+j)\cap(\Kz+j')=\emptyset$ if $j,j'\in\ZZ^d$ and $j\neq j'$.

We then define the {\em perforated domain}
\begin{equation}
E=\rr^d\setminus (\Kz+\ZZ^d).
\end{equation}

Let $\Omega$ be an open subset of $\rr^d$. For all $\e>0$ and $u\in L^2(\Omega\setminus \e E)$ we set
\begin{eqnarray}\label{def-Fe-0}
F_\e(u)&=&{1\over \e^{d+2}}\int_{\Omega\cap\e E}\int_{\Omega\cap\e E} a\Bigl({y-x\over\e}\Bigr)
(u(y)-u(x))^2dy \dx.
\end{eqnarray}
Equivalently, we will write
\begin{eqnarray}\label{def-Fe}\nonumber
F_\e(u)&=&{1\over \e^{d}}\int_{\Omega\cap\e E}\int_{\Omega\cap\e E} a\Bigl({y-x\over\e}\Bigr)
\Bigl({u(y)-u(x)\over\e}\Bigr)^2dy \dx\\ \nonumber
&=&{1\over \e^{d+2}}\int_{\Omega\cap\e E}\int_{\Omega\cap\e E-x} a\Bigl({\xi\over\e}\Bigr)
(u(x+\xi)-u(x))^2d\xi \dx
\\
&=&\int_{\Omega\cap\e E}\int_{{1\over\e}(\Omega\cap\e E-x)} a(\xi)
\Bigl({u(x+\e\xi)-u(x)\over\e}\Bigr)^2d\xi \dx,
\end{eqnarray}
this last version being useful when integrability properties of $a$ are used.

\bigskip
In the following we will prove extension, compactness and convergence properties for the functionals $F_\e$ under
hypothesis \eqref{bou_a}.
Note that if such a hypothesis is removed the functionals $F_\e$ may be `degenerate', as shown in Example \ref{degex}.

\subsection{Notation}
Unless otherwise stated $C$ denotes a generic strictly positive constant independent of the parameters of the problem taken into account.

$Q_T=[-{T\over2},{T\over2}]^d$ denotes the $d$-dimensional coordinate cube centered in $0$ and with side-length~$T$.

$\lfloor t\rfloor$ denotes the integer part of $t\in\rr$.

$\chi_A$ denotes the characteristic function of the set $A$.

For all $t>0$ we denote
$$\Omega(t)=\{x\in\Omega: {\rm dist}(x,\partial \Omega)>t\}.$$

\subsection{Preliminaries for `solid domains'}
In this section we consider `solid domains'; i.e., the case when the perforation is not present.
In the following we will extend functions from $\Omega\cap \e E$ to the whole $\Omega$ maintaining
the boundedeness of some convolution energy. In this way we will be able to use arguments for functions defined on solid domains.

\subsubsection{A Compactness Theorem}
Let $\Omega$ be an open set with Lipschitz boundary.  The following result, proved in \cite{2019BP} show that families of functions  that have bounded energies of the type \eqref{def-Fe-solid} is compact in $L^2_{\rm loc}(\Omega)$.


\begin{theorem}[compactness theorem]\label{t_comp}
Let $\Omega$ be an open set with Lipschitz boundary, and assume that for a family $\{w_\e\}_{\e>0}$, $w_\e\in L^2(\Omega)$,
the estimate  \begin{equation}\label{energ_bou}
F_\e^{k\e,r}(w_\e):=\int_{\Omega(k\e)}\int_{\{|\xi|\leq r\}}
\Bigl({w_\e(x+\e\xi)-w_\e(x)\over\e}\Bigr)^2d\xi \dx\le C
\end{equation}
is satisfied with some  $k>0$ and $r>0$.
Assume moreover that the family $\{w_\e\}$ is bounded in $L^2(\Omega)$.
Then for any sequence ${\e_j}$ such that $\e_j>0$ and $\e_j\to0$, as $j\to\infty$,  and for any open subset
$\Omega'\Subset\Omega$  the set $\{w_{\e_j}\}_{j\in\mathbb N}$ is relatively compact in $L^2(\Omega')$  and every its limit point is in $H^1(\Omega)$.
\end{theorem}

\subsubsection{Treatment of boundary data}
In this section we prove a classical lemma that allows to match boundary data. For future reference we prove it for general integrands $b_\e$ which only satisfy an estimate from above. In the following it will be applied to
\begin{equation}\label{be=a}
b_\e(x,y)= b(x,y)=\chi_{E}(x)\chi_{E}(y) a(x-y).
\end{equation}

With this requirement of generality in mind, let $b_\e:\Omega\times\Omega\to\rr$ be Borel functions satisfying
\begin{eqnarray}\label{ass-a}
0 \le b_\e(x,y)\le C {1\over (1+|x-y|)^{d+2+\kappa}}.
\end{eqnarray}
and set
\begin{eqnarray}\label{def-Fe-b}
F^b_\e(u)={1\over \e^{d+2}}\int_{\Omega}\int_{\Omega} b_\e\Bigl({x\over\e},{y\over\e}\Bigr)
(u(y)-u(x))^2\,dy \dx.
\end{eqnarray}

\def\Domain{\Omega}

\begin{proposition}[treatment of boundary values]\label{bvp}
Let $A$ be a bounded open set with Lipschitz boundary, let $v_\eta\to v$ in $L^2(A)$ with $v\in H^1(A)$.
For every $\delta>0$ there exist $v^\delta_\eta$ converging to $v$ in $L^2(A)$ such that
$$
v^\delta_\eta= v \hbox{ in } A\setminus A(\delta), \qquad v^\delta_\eta= v_\eta \hbox{ in } A(2\delta)
$$
and
$$
\limsup_{\eta\to 0} (F^b_\eta (v^\delta_\eta)- F^b_\eta(v_\eta))\le o(1)
$$
as $\delta\to 0$.
\end{proposition}

\begin{proof} With fixed $N\in\NN$ and  $k\in\{1,\ldots, N\}$ (to be determined below) let
$$
\phi(x)= \min\Biggl\{\max\Bigl\{0,{N\over \delta }\Bigl({\rm dist}(x,\partial A)-\delta\Bigl(1+{k-1\over N}\Bigr)\Bigr)\Bigr\},1\Biggr\},
$$
so that $\phi=0$ in $A\setminus A(\delta(1+{k-1\over N}))$,  $\phi=1$ in $A(\delta(1+{k\over N}))$, and $|\nabla\phi|\le {N\over\delta}$.

We define
$$
v^\delta_\eta=\phi \,v_\eta +(1-\phi)\,v .
$$
Upon writing
\begin{eqnarray*}
v^\delta_\eta(y)-v^\delta_\eta(x)&=&
\phi(y)(v_\eta(y)-v_\eta(x))+(1-\phi(y))(v (y)-v (x))\\
 &&+(\phi(y)-\phi(x))(v_\eta(x)-v (x)),
\end{eqnarray*}
we can estimate $F^b_\eta (v^\delta_\eta)$ by examining separately the sets
\begin{eqnarray*}
B_1&=&A\Bigl(\delta\Bigl(1+{k\over N}\Bigr)\Bigr)\times A\Bigl(\delta\Bigl(1+{k\over N}\Bigr)\Bigr) ,
\\
B_2&=&\Bigl(A\setminus A\Bigl(\delta\Bigl(1+{k-1\over N}\Bigr)\Bigr)\Bigr)\times \Bigl(A\setminus A\Bigl(\delta\Bigl(1+{k-1\over N}\Bigr)\Bigr)\Bigr)
\\
B_3&=&\Bigl(A\Bigl(\delta\Bigl(1+{k-1\over N}\Bigr)\Bigr)\setminus A\Bigl(\delta\Bigl(1+{k\over N}\Bigr)\Bigr)\times A
\\
B'_3&=& \Bigl(A\times \Bigl(\Bigl(A\Bigl(\delta\Bigl(1+{k-1\over N}\Bigr)\Bigr)\setminus A\Bigl(\delta\Bigl(1+{k\over N}\Bigr)\Bigr)\Bigr)\Bigr)\setminus B_3
\\
B_4&=& A\Bigl(\delta\Bigl(1+{k\over N}\Bigr)\Bigr)\times \Bigl(A\setminus A\Bigl(\delta\Bigl(1+{k-1\over N}\Bigr)\Bigr)\Bigr)
\\
B'_4&=& \Bigl(A\setminus A\Bigl(\delta\Bigl(1+{k-1\over N}\Bigr)\Bigr)\Bigr)\times A\Bigl(\delta\Bigl(1+{k\over N}\Bigr)\Bigr).
\end{eqnarray*}

We have
\begin{eqnarray*}&& \hskip-2cm
{1\over\eta^{d}}
\int_{B_1} b_\eta\Bigl({x\over\eta}, {y\over\eta}\Bigr)
\Big({v^\delta_\eta(y)-v^\delta_\eta(x)\over \eta}\Big)^2 \dx\,dy\\
&=&{1\over\eta^{d}}\int_{B_1} b_\eta\Bigl({x\over\eta}, {y\over\eta}\Bigr)
\Bigl({v_\eta(y)-v_\eta(x)\over\eta}\Bigr)^2 \dx\,dy\le F^b_\eta (v_\eta);
\end{eqnarray*}
\begin{eqnarray*}&& \hskip-2cm
{1\over\eta^{d}}\int_{B_2} b_\eta\Bigl({x\over\eta}, {y\over\eta}\Bigr)
\Bigl({v^\delta_\eta(y)-v^\delta_\eta(x)\over\eta}\Bigr)^2\dx\,dy\\
&=&{1\over\eta^{d}}\int_{B_2} b_\eta\Bigl({x\over\eta}, {y\over\eta}\Bigr)
\Bigl({v(y)-v(x)\over\eta}\Bigr)^2\dx\,dy
\\
&\le&{1\over\eta^{d}}\int_{A(2\delta)\times A(2\delta)} b_\eta\Bigl({x\over\eta}, {y\over\eta}\Bigr)
\Bigl({v(y)-v(x)\over\eta}\Bigr)^2\dx\,dy
\\
&\le&
\omega(|A(2\delta)|),
\end{eqnarray*}
with $\omega(s)\to 0$ as $s\to 0$.

We set $$
S_k=A\Bigl(\delta\Bigl(1+{k-1\over N}\Bigr)\Bigr)\setminus A\Bigl(\delta\Bigl(1+{k\over N}\Bigr)\Bigr)$$ and note that
\begin{eqnarray*}
&& \hskip-2cm{1\over\eta^d}\int_{S_k\times A} b_\eta\Bigl({x\over\eta}, {y\over\eta}\Bigr)
\Bigl({\phi(y)-\phi(x)\over\eta}\Bigr)^2 (v_\eta-v)^2 \dx\,dy\\
&\le&  {N^2\over\delta^2}{1\over\eta^{d+2}}\int_{S_k\times A} b_\eta\Bigl({x\over\eta}, {y\over\eta}\Bigr)
|x-y|^2 (v_\eta-v)^2 \dx\,dy\\
&\le& C{N^2\over\delta^2}\int_{S_k} (v_\eta-v)^2\dx.
\end{eqnarray*}
We then have
\begin{eqnarray*}&& \hskip-2cm
{1\over\eta^d}\int_{B_3} b_\eta\Bigl({x\over\eta}, {y\over\eta}\Bigr)
\Bigl({v^\delta_\eta(y)-v^\delta_\eta(x)\over\eta}\Bigr)^2\dx\,dy\\
&\le& C{1\over\eta^d}\int_{S_k\times A} b_\eta\Bigl({x\over\eta}, {y\over\eta}\Bigr)
\Bigl({v_\eta(y)-v_\eta(x)\over\eta}\Bigr)^2\dx\,dy
\\
&&+ C{1\over\eta^d}\int_{S_k\times A} b_\eta\Bigl({x\over\eta}, {y\over\eta}\Bigr)
\Bigl({v(y)-v(x)\over\eta}\Bigr)^2\dx\,dy
\\
&&+ C{N^2\over\delta^2}\int_{S_k} (v_\eta-v)^2\dx.
\end{eqnarray*}
We may now choose $k\in\{1,\ldots, N\}$ such that
\begin{eqnarray*}
 &&{1\over\eta^d}\int_{S_k\times A} b_\eta\Bigl({x\over\eta}, {y\over\eta}\Bigr)
\Bigl({v_\eta(y)-v_\eta(x)\over\eta}\Bigr)^2\dx\,dy
+ {1\over\eta^d}\int_{S_k\times A} b_\eta\Bigl({x\over\eta}, {y\over\eta}\Bigr)
\Bigl({v(y)-v(x)\over\eta}\Bigr)^2\dx\,dy
\\
&&\qquad\qquad\qquad \le {1\over N} (F^b_\eta(v_\eta)+F^b_\eta(v)),
\end{eqnarray*}
so that
$$
{1\over\eta^d}\int_{B_3} b_\eta\Bigl({x\over\eta}, {y\over\eta}\Bigr)
\Bigl({v^\delta_\eta(y)-v^\delta_\eta(x)\over\eta}\Bigr)^2\dx\,dy
\le {C\over N}+C{N^2\over\delta^2}\int_{S_k} (v_\eta-v)^2\dx.
$$
The same argument proves the same estimate for the integral on $B'_3$.

As for $B_4$, note that for $(x,y)\in B_4$ we have
$$
|y-x|\ge {\delta\over N},
$$
so that, using the growth assumption \eqref{ass-a} on $b_\eta$, we obtain
\begin{eqnarray*}
{1\over\eta^d}\int_{B_4} b_\eta\Bigl({x\over\eta}, {y\over\eta}\Bigr)
\Bigl({v^\delta_\eta(y)-v^\delta_\eta(x)\over\eta}\Bigr)^2\dx\,dy
&\le& C {N^{2+d+\kappa}\over \delta^{2+d+\kappa}}\eta^k.
\end{eqnarray*}
The estimate for the contribution of $B'_4$ is completely analogous.

Gathering all estimates above we have
$$
F^b_\eta(v^\delta_\eta)- F^b_\eta(v_\eta)\le
\omega(|A(2\delta)|)+{C\over N}+C{N^2\over\delta^2}\int_{A} (v_\eta-v)^2\dx
+C {N^{2+d+\kappa}\over \delta^{2+d+\kappa}}\eta^k
$$
Letting $\eta\to 0$ and using the arbitrariness of $N$ we obtain the claim.
\end{proof}

\begin{remark}\label{bvp-rem}\rm We may generalize the proposition above by fixing as boundary data, instead of a fixed $v\in H^1(A)$, a sequence $w_\eta$ converging weakly to $v$ in $H^1(A)$ and
satisfying
$$
\limsup_{\eta\to 0}\int_{U}\int_{\{\xi: x+\eta\xi\in U\}} {1\over (1+|\xi|)^{d+2+\kappa}}
\Bigl({w_\eta(x+\eta\xi)-w_\eta(x)\over\eta}\Bigr)^2d\xi \dx\le \omega (|U|)
$$
for all open sets $U$, where $\omega(t)\to 0$ as $t\to 0$. In this case, by defining
$$
v^\delta_\eta=\phi \,v_\eta +(1-\phi)\,w_\eta .
$$
in the proof above, we can obtain that the $v^\delta_\eta$ converging to $v$ in $L^2(A)$ satisfy
$$
v^\delta_\eta= w_\eta \hbox{ in } A\setminus A(\delta), \qquad v^\delta_\eta= v_\eta \hbox{ in } A(2\delta).
$$
\end{remark}

\section{An extension theorem}
The following result states that any function $u$ defined on the perforated domain $\Omega\cap \e E$ can be extended to a function defined on the whole domain $\Omega$ such that, upon restricting to the interior of the domain (more precisely, to points at distance { greater than} some multiple of $\e$) a reference convolution energy of the extension is bounded by $F_\e(u)$ and the $L^2$-norm of the extension is bounded by the $L^2$-norm of $u$ on the perforated domain. This result mirrors the analog where $F_\e$ is the Dirichlet energy on the
perforated domain and the reference convolution energy is replaced with the Dirichlet energy on the solid domain \cite{ACDP}.

\begin{definition}\label{uLs} We say that an open subset $\cal F$ of $\rr^d$ has {\em  a uniformly Lipschitz countinuous boundary} if
there exist constants  $L>0$ and $\rho_1, \,\rho_2>0$    such that for any point $x\in\partial\mathcal{F} $  there exists a set $\mathbb S$ which,  up to translation by $x$ and rotation, is of the form $(-\rho_1,\rho_1)^{d-1}\times(-\rho_2,\rho_2)$ such that  $\mathbb S\cap \mathcal{F}$ is the sub-graph of a $L$-Lipschitz function defined on $(-\rho_1,\rho_1)^{d-1}$.
\end{definition}

Notice that we can assume without loss of generality that the $L$-Lipschitz function in the last definition takes
on values in the interval  $(-\frac{\rho_2}2,\frac{\rho_2}2)$.

\begin{theorem}[extension theorem]\label{t_ext}
Let $E$ be an open subset of $\rr^d$ with the following properties

  $1$. $\mathbb R^d\setminus E$ is a union of bounded open sets in $\mathbb R^d$;

 $2$. the diameters of these sets are uniformly bounded;

  $3$. the distance between any two distinct sets is bounded from below by a positive constant;

  $4$. the boundary of the set $E$ is uniformly Lipschitz continuous in the sense of Definition \ref{uLs}.

Let $\Omega$ be an open set with Lipschitz boundary. For each $r_0>0$ there exist $k>0$ and $r>0$ such that
for all $u\in L^2(\Omega\cap \e E)$ there exists $v\in L^2(\Omega)$ such that
\begin{equation}\label{odin}
v=u\ \ \hbox{\rm on } \Omega\cap\e E,
\end{equation}
\begin{equation}\label{dva}
\int_{\Omega(k\e)\times\Omega(k\e)\cap\{|x-y|\le\e r\}}
\Bigl({v(y)-v(x)\over\e}\Bigr)^2dy \dx\le
C \int_{(\Omega\cap\e E)\times(\Omega\cap\e E)\cap\{|x-y|\le\e r_0\}}
\Bigl({u(y)-u(x)\over\e}\Bigr)^2dy \dx
\end{equation}
and
\begin{equation}\label{tri}
\int_{\Omega(k\e)}|v|^2\dx\le C\int_{\Omega\cap\e E}|u|^2\dx.
\end{equation}
\end{theorem}

Note that if \eqref{bou_a} holds then \eqref{dva} implies that
\begin{equation}\label{dva-energy}
\int_{\Omega(k\e)}\int_{\{|\xi|\le r\}}
\Bigl({v(x+\e\xi)-v(x)\over\e}\Bigr)^2d\xi \dx\le
C F_\e(u)
\end{equation}

Our arguments rely on the following statement.

\begin{lemma}\label{ene_locali} Let $A$ be a connected bounded Lipschitz set.
For any $r>0$ there exists a constant $c_r>0$ such that the following inequality holds
\begin{equation}\label{ineq_loc1}
\int_{A\times A}(u(\eta)-u(\xi))^2\,d\xi \,d\eta\leq c_r
\int_{(A\times A)\cap\{
(\xi,\eta): |\xi-\eta|\leq r\}}(u(\eta)-u(\xi))^2\,d\xi\, d\eta.
\end{equation}
If   $L$, $\rho_1,\rho_2$ are constants as in Definition \ref{uLs} then
the constant $c_r$ depends on $A$ only through its diameter and such constants.
\end{lemma}


\begin{proof}[Proof of Lemma {\rm\ref{ene_locali}}]
Since for any function $u$ the integral on the right-hand side of  \eqref{ineq_loc1} is an increasing function
of  $r$, it is sufficient to prove \eqref{ineq_loc1} for  $r$ positive and small enough.

Since $A$ has a Lipschitz boundary and is connected, with fixed $r>0$ there exists $r_1\in (0, \frac12 {r})$
 and $\nu\in(0,1]$ that only depends on the Lipschitz constant of $A$ such that for any two points $\eta',\,\eta''\in A$
there is a discrete path from $\eta'$ to $\eta''$; i.e., a set of points $\eta'=\eta_0,\,\eta_1, \ldots, \eta_N,\,\eta''=\eta_{N+1}$,
 that possesses the following properties:
\begin{itemize}
  \item [(a)] $|\eta_{j+1}-\eta_j|\leq r_1$ for $j=0,\,1,\dots, N$;
  \item [(b)] for all $j=1,\,\ldots, N$ the ball $B_{\nu r_1}(\eta_j)=\{\eta\in\mathbb R^d\,:\,|\eta-\eta_j|\leq \nu r_1\}$
  is contained in $A$;
  \item [(c)]there exists  $\bar N=\bar N(r_1, \mathrm{diam}(A))$ such that $N\leq \bar N$ for all $\eta'$, $\eta''\in A$.
\end{itemize}
Indeed, since $A$ is a bounded Lipschitz set, it has a uniformly Lipschitz continuous boundary. Then there exists
a constant $r_2=r_2(L,\rho_1,\rho_2,r)>0$ such that $r_2<\frac12 r$, $r_2<\frac1{2d}\min(\rho_1,\rho_2)$,
and the set $A_{r_2}=\{x\in A\,:\,\mathrm{dist}(x,\partial A)>r_2\}$ is connected.  We  choose 
$r_1=\frac{r_2}{8(L+1)}$ and denote $Z_{A}=\{z\in \frac{r_1}{\sqrt{d}}\mathbb Z^d\,:\, z\in A_{r_2}\}$.
By construction  $B_{r_1}(x)\subset A$ for any $x\in A_{r_2}$, and for any $z_1$ and $z_2$ in $Z_A$ there exists a path $z_1=\eta_1,\ldots, \eta_N=z_2$ in $Z_A$ such that $|\eta_{j+1}-\eta_j|\leq r_1$ and
$N\leq \big(\frac {\mathrm{diam}(A)}{r_1}\big)^d$.  Also, by construction, for any $x\in A\setminus A_{r_2}$
there exists a path $x=\tilde\eta_0,\ldots,\tilde\eta_{\tilde N}$ such that $\tilde\eta_{\tilde N}\in Z_A$,
$|\tilde\eta_{j+1}-\tilde\eta_j|\leq r_1$, 
$\tilde N\leq 16(L+1)d$, and $B_{\frac{r_1}{2(L+1)}}(\tilde\eta_j)\subset A$ for all $j=1,\ldots, \tilde N$.
This implies the existence of a path that has properties $(a)$--$(c)$.

Writing
$$
u(\xi_0)-u(\xi_{N+1})=u(\xi_0)-u(\xi_1)+u(\xi_1)-\ldots -u(\xi_N)+u(\xi_N)-u(\xi_{N+1}),
$$
where $\xi_j$ denotes a point in $B_{\nu r_1}(\eta_j)$ for $1\le j\le N$, we get
\begin{eqnarray*}&&
\int_{(A\cap B_{\nu r_1}(\eta'))\times
(A\cap B_{\nu r_1}(\eta''))} (u(\xi_0)-u(\xi_{N+1}))^2\,d\xi_0 d\xi_{N+1}
\\
&=&  (\nu r_1)^{-dN} \int_{B_{\nu r_1}(\eta_1)}\dots\int_{B_{\nu r_1}(\eta_N)} \int_{(A\cap B_{\nu r_1}(\eta'))\times
(A\cap B_{\nu r_1}(\eta''))}\big\{u(\xi)-u(\xi_1)+u(\xi_1)-\ldots\\
&&\qquad\qquad\qquad\qquad\qquad\qquad\qquad\qquad -u(\xi_N)+u(\xi_N)-u(\eta)\big\}^2
\,d\xi_0\, d\xi_{N+1} \,d\xi_N\ldots d\xi_1\\
&\le& (N+1)(\nu r_1)^{-dN}\int_{A\cap B_{\nu r_1}(\eta_0)}\dots\int_{A\cap B_{\nu r_1}(\eta_{N+1})}\sum_{j=1}^{N+1}(u(\xi_j)-u(\xi_{j-1}))^2
d\xi_{N+1}\ldots d\xi_0\\&=&
 C(N+1)\sum_{j=1}^{N+1}\int_{(A\cap B_{\nu r_1}(\eta_{j}))\times (A\cap B_{\nu r_1}(\eta_{j-1}))}(u(\xi_j)-u(\xi_{j-1}))^2
d\xi_{j}\,d\xi_{j-1}\\
&\le&C
(\bar N+1)^2
\int_{(A\times A)\cap\{
(\xi,\eta): |\xi-\eta|\leq r\}}(u(\eta)-u(\xi))^2\,d\xi\, d\eta.
\end{eqnarray*}
 Covering $A$ with a finite number of balls of radius $\nu r_1$  and summing up the last inequality
 over all pairs of these balls gives  the desired estimate \eqref{ineq_loc1}.
\end{proof}


\def\K{{\bf K}}

\begin{proof}[Proof of Theorem {\rm\ref{t_ext}}]
We apply our arguments separately to each connected component of $\rr^d\setminus E$.
With fixed $\tau>0$ chosen below we consider a connected component $\K$ of $\rr^d\setminus E$, and set  $$A:=\{\xi\in \mathbb R^d\setminus \K\,:\, \mathrm{dist}(\xi,\partial \K)<\tau\}\hbox{ and }
A^\star:=\{\xi\in  \K\,:\, \mathrm{dist}(\xi,\partial \K)<\tau\}.$$

Since $\K$ is bounded and Lipschitz,  we may fix $\tau>0$ small enough and an invertible mapping $\mathcal{R}$ from  $A$
to $A^\star$ such that
$$
\frac12|\mathcal{R}(\xi')-\mathcal{R}(\xi'')|\leq |\xi'-\xi''|\leq 2|\mathcal{R}(\xi')-\mathcal{R}(\xi'')|
$$
for all $\xi',\,\xi''\in A$.
 Slightly abusing the notation we call this mapping  a {\sl reflection}.  In what follows for the sake of brevity we use the notation
$\xi_{\mathcal{R}}=\mathcal{R}^{-1}(\xi)$ for $\xi\in A^\star$.

 We set
$$
\bar u_A=\frac1{|A^\star|}\int_{A^\star}u(\mathcal{R}^{-1}(\xi))\,d\xi.
$$
Let $\varphi$ be a $C^\infty$ function such that $0\leq\varphi\leq 1$,   $\varphi=1$ in $A$ and in a neighbourhood
of $\partial \K$,  $\varphi=0$ in a neighbourhood of $\partial A^\star\setminus\partial \K$.

We define $v(\xi)$ as follows
$$
v(\xi)=\left\{\begin{array}{ll}
u(\xi) & \hbox{if }\xi\in A\\
\varphi(\xi)u(\mathcal{R}^{-1}(\xi))+(1-\varphi(\xi))\bar u_A& \hbox{if }\xi\in A^\star \\
\bar u_A& \hbox{if }\xi\in \K\setminus A^\star.
\end{array}
\right.
$$
Letting $k=\mathrm{diam}(Q_1)=\sqrt{d}$ and  $r=\min(r_0, k,\tau)$, we have
\begin{equation}\label{nstar_nstar}
\int_
{(A\times A)\cap\{|\xi-\eta|\leq r\}}
(v(\eta)-v(\xi))^2\,d\xi d\eta=
\int_
{(A\times A)\cap\{|\xi-\eta|\leq r\}}
(u(\eta)-u(\xi))^2\,d\xi d\eta
\end{equation}
and
\begin{eqnarray}\label{star_nstar}\nonumber
\int_
{(A\times A^\star)\cap\{|\xi-\eta|\leq r\}}
\!\!\!\!\!\!\!\!(v(\eta)-v(\xi))^2\,d\xi d\eta&\leq&
\int_{A\times A}
(u(\eta)-u(\zeta))^2\left|\frac{\partial\mathcal{R}(\zeta)}{\partial\zeta}\right|\,d\zeta d\eta
\\
&\le&
C_{\mathcal{R}}
\int_{A\times A}
(u(\eta)-u(\zeta))^2\,d\zeta d\eta.
\end{eqnarray}
Here we have used the fact that the Jacobian $\Big|\frac{\partial\mathcal{R}(\zeta)}{\partial\zeta}\Big|$ is a bounded function:
$\Big|\frac{\partial\mathcal{R}(\zeta)}{\partial\zeta}\Big|\leq C_{\mathcal{R}}$.\\
Next, taking into account the relation $$v(\xi)-v(\eta)=(\varphi(\xi)-\varphi(\eta))(\bar u_A-u(\xi_{\mathcal{R}}))+\varphi(\eta)
(u(\xi_{\mathcal{R}})-u(\eta_{\mathcal{R}}))
\hbox{ if }\eta\in A^\star,\,\xi\in A^\star
$$ we obtain
\begin{eqnarray*}&&\hskip-2cm
\int\limits_
{(A^\star\times A^\star)\cap\{|\xi-\eta|\leq r\}}
(v(\eta)-v(\xi))^2\,d\xi d\eta\\
&\leq&
\int_{A^\star\times A^\star}
(\bar u_A-u(\xi_{\mathcal{R}}))^2\,d\xi d\eta+\int_{A^\star\times A^\star
}(u(\eta_{\mathcal{R}})-u(\xi_{\mathcal{R}}))^2\,d\xi d\eta.
\end{eqnarray*}
Since $\bar u_A$ is the average of the function $u(\xi_{\mathcal{R}})$ over $A^\star$, then
$$
\int_{A^\star\times A^\star
}(\bar u_A-u(\xi_{\mathcal{R}}))^2\,d\xi d\eta =\frac12\int_{A^\star\times A^\star
}(u(\eta_{\mathcal{R}})-u(\xi_{\mathcal{R}}))^2\,d\xi d\eta.
$$
This yields
\begin{equation}
\label{star_star}
\int_
{( A^\star\times A^\star)\cap\{|\xi-\eta|\leq r\}}
(v(\eta)-v(\xi))^2\,d\xi d\eta\leq C^2_{\mathcal{R}}\int_{
A\times A}(u(\eta)-u(\xi))^2\,d\xi d\eta.
\end{equation}
Finally,
\begin{eqnarray*}
\int_{A^\star}\int_{\{\xi\in \K\setminus A^\star:|\xi-\eta|\leq r\}}
(v(\eta)-v(\xi))^2\,d\xi d\eta &\le&\int_{\K\setminus A^\star}\int_{A^\star}
(\varphi(\xi))^2(\bar u_A)-u(\xi_{\mathcal{R}}))^2\,d\xi\,d\eta
\\
&\leq& \big|\K\setminus A^\star\big|\int_{A^\star}
(\bar u_A)-u(\xi_{\mathcal{R}}))^2\,d\xi\\
&\le& C_{\mathcal{R}}\frac{\big|\K\setminus A^\star\big|}{|A^\star|}
\int_{A\times A}(u(\eta)-u(\xi))^2\,d\xi d\eta.
\end{eqnarray*}
Combining the last inequality with \eqref{nstar_nstar},  \eqref{star_nstar} and \eqref{star_star} we conclude that
\begin{equation}\label{ene_bou2}
  \int_{((\K\cup A)\times( \K\cup A))\cap
\{|\xi-\eta|\leq r\}}
(v(\eta)-v(\xi))^2\,d\xi d\eta \leq C \int_{A\times A}
(u(\eta)-u(\xi))^2\,d\xi d\eta
\end{equation}
We may now apply Lemma \ref{ene_locali}.
By  \eqref{ineq_loc1} we obtain

\begin{eqnarray*}
  \int_{((\K\cup A)\times( \K\cup A))\cap
\{|\xi-\eta|\leq r\}}
(v(\eta)-v(\xi))^2\,d\xi d\eta &\le& C \int_{(A\times A)\cap
\{|\xi-\eta|\leq r\}}
(u(\eta)-u(\xi))^2\,d\xi d\eta.
\end{eqnarray*}
After rescaling, this inequality reads
\begin{equation}\label{ene_bou3}
  \int_{(\varepsilon(\K\cup A)\times
  \varepsilon(\K\cup A))\cap\{|\xi-\eta|\leq \varepsilon r\}}
(v(\eta)-v(\xi))^2\,d\xi d\eta
\leq C_1 \int_{(\varepsilon A\times \varepsilon A)\cap\{|\xi-\eta|\leq \varepsilon r\}}
(u(\eta)-u(\xi))^2\,d\xi d\eta
\end{equation}
Summing up the last inequality over all the inclusions
in $\Omega(k\varepsilon)$, we obtain  \eqref{odin} and \eqref{dva}.  Inequality \eqref{tri} is a straightforward
consequence of the definition of $v$.  
\end{proof}

\section{Homogenization}\label{hom}
We now state and prove a homogenization result using $\Gamma$-convergence \cite{GCB,DM,Handbook}.
To this end we have first to specify the convergence with respect to which we compute the $\Gamma$-limit.
The choice is driven by the compactness Theorem \ref{t_comp} coupled with the extension result in the previous section.

\begin{definition}\label{de-con}
We say that a sequence $\{u_\e\}$ in $L^2(\Omega\cap \e E)$ converges to $u\in L^2(\Omega)$ if
there exists a sequence $\{\tilde u_\e\}$ in $L^2(\Omega)$ of extended functions with $\tilde u_\e= u_\e$
in $\Omega\cap \e E$ converging to $u\in L^2_{\rm loc}(\Omega)$.
\end{definition}

Note that the limit $u$ does not depend on the extensions $\{\tilde u_\e\}$ since it is also the weak $L^2_{\rm loc}(\Omega)$
limit of the sequence $|Q_1\cap E|^{-1}\chi_{\e E} u_\e$.

\begin{theorem}\label{thm_homogee}
Let $\Omega$ be an open set with Lipschitz boundary, and let $F_\e$ be given by \eqref{def-Fe-0}.
Then $F_\e$ $\Gamma$-converge with respect to the convergence $u_\e\to u$ 
 in Definition {\rm\ref{de-con}} to the functional
\begin{equation}\label{homfuneq}
F_{\rm hom}(u)=\int_\Omega \langle A_{\rm hom}\nabla u,\nabla u\rangle\dx
\end{equation}
with domain $H^1(\Omega)$, where $A_{\rm hom}$ is a symmetric matrix
given by the cell-problem formula
\begin{equation}\label{cpformula}
 \langle A_{\rm hom}z,z\rangle=\inf\Bigl\{\int_{(0,1)^d\cap E}\int_{E} a(x-y)(w(x)-w(y))^2 dx\,dy:
 w(y)-\langle z,y\rangle\hbox{ is $1$-periodic}\Bigr\}.
\end{equation}
\end{theorem}

From this theorem, thanks to the compactness results of the previous sections and the validity of Poincar\'e inequalities (see \cite[Corollary 4.2]{2019BP}) we deduce the convergence of minimum problems with forcing terms and/or boundary data.

The example below shows that if condition \eqref{bou_a} is not satisfied then the homogenization theorem may not hold.
\begin{example}\label{degex} \rm
We fix $\delta<{1\over 4}$, consider the $1$-periodic connected set $E$ in $\rr^2$ such that
$$
E\cap Q_1= Q_1\setminus Q_{1-\delta},
$$
and set
$$
a=\chi_{\bigl((0,{1\over 2})+Q_\delta\bigr)}.
$$
Note that the set
$$
\{x\in\rr^2: \hbox{there exists }\xi \hbox{ such that } a(x+\xi)\neq 0\}= E\cap \bigl(E+\bigl(0,\hbox{${1\over 2}$}\bigr)+Q_\delta\bigr)$$
is contained in the collection of horizontal stripes
$\{x\in\rr^2: {\rm dist}(x_2,\ZZ)<2\delta\}$.
This implies that functions which are constant in each of those stripes have zero energy, and as a consequence the $\Gamma$-limit is zero on all $L^2$ functions $u=u(x_1,x_2)=u(x_2)$ independent of the first variable and in particular it cannot be represented as in
\eqref{homfuneq}. \end{example}

We subdivide the proof of Theorem \ref{thm_homogee} into a lower bound (Proposition \ref{lbound}) and an upper bound
(Proposition \ref{ubound}). The proof of the lower bound is itself split into several steps.
Moreover formula \eqref{cpformula} is obtained by remarking that the limit energy density is a quadratic form (Remark \ref{quaform}).

\begin{proposition}[lower bound]\label{lbound}
For all sequences $u_\e\to u$ we have $\liminf\limits_{\e\to 0} F_\e(u_\e)\ge F_{\rm hom}(u)$.
\end{proposition}

\begin{proof}
We fix a sequence $u_\e\to u$ with bounded $F_\e(u_\e)$. Upon using the extension Theorem \ref{t_ext} and the compactness Theorem \ref{t_comp} we may suppose that
$u_\e\to u$ in $L^2_{\rm loc}(\Omega)$ and that $u\in H^1(\Omega)$.

In order to prove the lower bound we use a variation of the Fonseca-M\"uller blow-up technique \cite{FM}. In the proof below we describe the general outline of the method (see also \cite{BMS} for more details on the application of the blow-up technique to homogenization). The proof of the claims which are particular of our problem are postponed to separate propositions.

We define the measures
$$
\mu_\e(A)=\int_{A\cap \e E}\int_{\{\xi: x+\e\xi\in \Omega\cap \e E\}}
a(\xi)\Bigl({u_\e(x+\e\xi)-u_\e(x)\over \e}\Bigr)^2d\xi\dx.
$$
Since $\mu_\e(A)=F_\e(u_\e)$, these measures are equibounded, and we may suppose that they converge weakly$^*$ to some measure $\mu$. We now fix an arbitrary Lebesgue point $x_0$ for $u$ and $\nabla u$, and set $z=\nabla u(x_0)$. The lower-bound inequality is proved if we show that
\begin{equation}
{d\mu\over dx}(x_0)\ge \langle A_{\rm hom}z,z\rangle.
\end{equation}

Upon a translation argument it is not restrictive to suppose that $x_0=0$, that
$0$ is a Lebesgue point of all $u_\e$ (upon passing to a subsequence), and that $u_\e(0)=u(0)=0$.
We note that for almost all $\rho>0$ we have $\mu_\e(Q_\rho)\to \mu(Q_\rho)$.
Since
$$
{d\mu\over dx}(0)= \lim_{\rho\to 0^+}{\mu(Q_\rho)\over\rho^d},
$$
we may choose (upon passing to a subsequence) $\rho=\rho_\e>\!>\e$ such that
$$
{d\mu\over dx}(0)= \lim_{\e\to 0^+}{\mu_\e(Q_\rho)\over\rho^d}.
$$
We trivially have
\begin{eqnarray*}
\mu_\e(Q_\rho)&\ge&\int_{Q_\rho\cap \e E}\int_{\{\xi: x+\e\xi\in Q_\rho\cap \e E\}}
a(\xi)\Bigl({u_\e(x+\e\xi)-u_\e(x)\over \e}\Bigr)^2d\xi\dx\\
&\ge&\int_{Q_{\e\lfloor\rho/\e\rfloor}\cap \e E}\int_{\{\xi: x+\e\xi\in Q_{\e\lfloor\rho/\e\rfloor}\cap \e E\}}
a(\xi)\Bigl({u_\e(x+\e\xi)-u_\e(x)\over \e}\Bigr)^2d\xi\dx.
\end{eqnarray*}
In order to ease the notation we will directly suppose that
\begin{equation}{\rho\over\e}\in \NN, \end{equation}
so that our claim is proven if we show that
\begin{equation}\label{main_claim}
\lim_{\e\to0^+}
{1\over \rho^d}\int_{Q_\rho\cap \e E}\int_{\{\xi: x+\e\xi\in Q_\rho\cap \e E\}}
a(\xi)\Bigl({u_\e(x+\e\xi)-u_\e(x)\over \e}\Bigr)^2d\xi\dx
\ge \langle A_{\rm hom}z,z\rangle.
\end{equation}

We now change variables and set
$$
v_\e(y)= {u_\e(\rho y)\over \rho} \hbox{ for } y\in Q_1
$$
Note that, since ${u(\rho y)\over \rho}$ converges to $\langle z,y\rangle$ as $\rho\to 0$ as we have assumed that $u(0)=0$,
and we also have assumed that  $u_\e(0)=0$, we may choose $\rho=\rho_\e$ above so that
$$
v_\e \to \langle z,y\rangle\hbox{ in } L^2(Q_1).
$$

\noindent
{\bf Claim 1.} For all $\delta>0$ there exists a sequence $v^\delta_\e$ such that $v^\delta_\e(y)=\langle z,y\rangle$ on $Q_1\setminus Q_{1-\delta}$ and
\begin{eqnarray*}
&&{1\over\rho^d}\int_{Q_\rho\cap \e E}\int_{\{\xi: x+\e\xi\in Q_\rho\cap \e E\}}
a(\xi)\Bigl({u_\e(x+\e\xi)-u_\e(x)\over \e}\Bigr)^2d\xi\dx
\\
&\ge&
\int_{Q_1\cap {\e\over\rho} E}\int_{\{\xi: y+{\e\over\rho}\xi\in Q_1\cap {\e\over\rho} E\}}
a(\xi)\Bigl({v^\delta_\e(y+{\e\over\rho}\xi)-v^\delta_\e(y)\over \e/\rho}\Bigr)^2d\xi\,dy + o(1)
\end{eqnarray*}
as $\delta\to 0$ uniformly in $\e$. The proof of this claim is obtained by Proposition \ref{bvp}, applied with
$b_\e$ as in \eqref{be=a}, $A=Q_1$ and
$\eta=\e/\rho$.

Thanks to Claim 1, \eqref{main_claim} is proved if we show that
\begin{eqnarray}\nonumber
&&\liminf_{\e\to 0}\min\Bigl\{\int_{Q_1\cap {\e\over\rho} E}\int_{\{\xi: y+{\e\over\rho}\xi\in Q_1\cap {\e\over\rho} E\}}
a(\xi)\Bigl({v(y+{\e\over\rho}\xi)-v(y)\over \e/\rho}\Bigr)^2d\xi\,dy: v(y)=\langle z,y\rangle\hbox{ on }Q_1\setminus Q_{1-\delta}\Bigr\}
\\
&\ge&\langle A_{\rm hom}z,z\rangle
\end{eqnarray}

We actually have that the liminf above may be turned into a limit and equality holds. This is stated in the two following claims, in which we use the change of variables $T={\rho\over\e}$. Recall that $T\in\NN$.

\medskip\noindent
{\bf Claim 2.} There exists the limit
\begin{eqnarray}f_{\rm hom}(z)&=&\nonumber
\lim_{T\in \NN, \ T\to+\infty}{1\over T^d}\min\Bigl\{\int_{Q_T\cap E}\int_{\{\xi: y+\xi\in Q_T\cap E\}}
a(\xi)(w(y+\xi)-w(y))^2d\xi\,dy: \\
&& \hskip2cm w(y)=\langle z,y\rangle\hbox{ on }Q_T\setminus Q_{(1-\delta)T} \hbox{ for some }\delta>0\Bigr\}\,.
\end{eqnarray}
The proof of this asymptotic homogenization formula is given in Proposition \ref{def-fhom}.

\medskip\noindent
{\bf Claim 3.}
We have $f_{\rm hom}(z)\ge f_0(z)$, where $f_0(z)$ is given by \eqref{cpformula}.

\noindent
The proof of this claim is given by Proposition \ref{boundper}.
\end{proof}

\begin{proposition}\label{small-ext}
Let $N\in\NN$ and let a function $w$ be such that
$(w-\langle z,y\rangle)$ is $N$-periodic in each coordinate direction  and
$$
w(y)=\langle z,y\rangle\hbox{ on }Q_N\setminus Q_{(1-\delta)N} \hbox{ for some }\delta>0.
$$
Then
\begin{eqnarray}\label{cubeN-est}\nonumber
&&\int_{Q_N\cap E}\int_{\{\xi: y+\xi\in E\setminus Q_N\}}
a(\xi)(w(y+\xi)-w(y))^2d\xi\,dy\\ \nonumber
&\le& C{1\over \delta^{2+d+\kappa} N^\kappa}
\int_{Q_N\cap E}\int_{\{\xi: y+\xi\in Q_N\cap E\}}
a(\xi)(w(y+\xi)-w(y))^2d\xi\,dy \\
&&\qquad+ CN^{d-1}\delta|z|^2+ C{N^{d-\kappa}\over \delta^{2+d+\kappa}},
\end{eqnarray}
with $C$ independent of $N$ for $N$ large.
\end{proposition}

\begin{proof}
Using the extension Theorem \ref{t_ext} we may suppose that $w$ is defined in the whole $\rr^d$ and it satisfies
\begin{eqnarray}\label{ext-est}\nonumber
&&\int_{Q_N}\int_{\{\xi: \xi\in Q_1, \ y+\xi\in Q_N\}}
(w(y+\xi)-w(y))^2d\xi\,dy\\
&\le& C \int_{Q_N\cap E}\int_{\{\xi: y+\xi\in E\cap Q_N\}}
a(\xi)(w(y+\xi)-w(y))^2d\xi\,dy,
\end{eqnarray}
with $C$ independent of $N$ for $N$ large. Note that we may suppose that the extension estimate holds with $\xi\in Q_1$ upon a change of variables, and $N\delta>1$.

We have to estimate
\begin{eqnarray*}
\sum_{j\in \ZZ^d\setminus\{0\}}
\int_{Q_N\cap E}\int_{\{\xi: y+\xi\in E\cap(Nj+ Q_N)\}}
a(\xi)(w(y+\xi)-w(y))^2d\xi\,dy.
\end{eqnarray*}

We subdivide the estimate into several computations.
We first consider the indices for the $3^d-1$ cubes neighbouring $Q_N$; i.e., $j$ with $\|j\|_\infty=1$.
For such indices, we separately estimate the energy of the interactions with the `interior' of the cube $Q_N$.
Note that for such interactions $|\xi|\ge\delta N$, Using the decay assumption \eqref{ass-a} and the periodicity of $w(y)-\langle z,y\rangle$, we get
\begin{eqnarray*}
&&\int_{Q_{(1-\delta)}N\cap E}\int_{\{\xi: y+\xi\in E\cap(Nj+ Q_N)\}}
a(\xi)(w(y+\xi)-w(y))^2d\xi\,dy\\
&=&\int_{Q_{(1-\delta)}N\cap E}\int_{\{\xi: |\xi|\ge\delta N, \, y+\xi\in E\cap(Nj+ Q_N)\}}
a(\xi)(w(y+\xi)-w(y))^2d\xi\,dy\\
&\le&{C\over (\delta N)^{2+d+\kappa}}
\int_{Q_{(1-\delta)N}\cap E}\int_{\{\xi: y+\xi-jN\in E\cap Q_N\}}
(w(y+\xi-jN)-w(y)+N\langle z,j\rangle)^2d\xi\,dy\\
\\
&\le&{C\over (\delta N)^{2+d+\kappa}}
\int_{Q_{N}\cap E}\int_{\{\xi: y+\eta\in E\cap Q_N\}}
\Bigl((w(y+\eta)-w(y))^2+N^2|z|^2\Bigr)d\eta\,dy
\end{eqnarray*}
We now use the estimate
\begin{eqnarray}\label{est-inter}
&& \nonumber
\int_{Q_{N}\cap E}\int_{\{\eta: y+\eta\in E\cap Q_N\}}
(w(y+\eta)-w(y))^2d\eta\,dy
\\ \nonumber
&\le&\int_{Q_{N}}\int_{\{\eta\in Q_N: y+\eta\in Q_N\}}
(w(y+\eta)-w(y))^2d\eta\,dy
\\ \nonumber
&=&\int_{Q_{N}}\int_{\{\eta\in Q_N: y+\eta\in Q_N\}}
\Biggl(\sum_{k=1}^N \Bigl(w\Bigl(y+\eta{k\over N}\Bigr)-w\Bigl(y+\eta{k-1\over N}\Bigr)\Bigr)\Biggr)^2d\eta\,dy
\\ \nonumber
&\le&\int_{Q_{N}}\int_{\{\eta\in Q_N: y+\eta\in Q_N\}} N
\sum_{k=1}^N \Bigl(w\Bigl(y+\eta{k\over N}\Bigr)-w\Bigl(y+\eta{k-1\over N}\Bigr)\Bigr)^2d\eta\,dy
\\ \nonumber
&\le&N^{2+d}\int_{Q_{N}}\int_{Q_1} (w(y+\xi)-w(y))^2d\xi\,dy
\\
&\le&N^{2+d}C \int_{Q_N\cap E}\int_{\{\xi: y+\xi\in E\cap Q_N\}}
a(\xi)(w(y+\xi)-w(y))^2d\xi\,dy
\end{eqnarray}
by \eqref{ext-est}. Plugging this estimate in the previous one we then have
\begin{eqnarray}\label{est-1}\nonumber
&&\int_{Q_{(1-\delta)N}\cap E}\int_{\{\xi: y+\xi\in E\cap(Nj+ Q_N)\}}
a(\xi)(w(y+\xi)-w(y))^2d\xi\,dy\\
&\le&{C\over \delta^{2+d+\kappa} N^\kappa} \int_{Q_N\cap E}\int_{\{\xi: y+\xi\in E\cap Q_N\}}
a(\xi)(w(y+\xi)-w(y))^2d\xi\,dy +C{N^{d-\kappa}\over \delta^{2+d+\kappa}}.
\end{eqnarray}
A completely analogous argument shows that
\begin{eqnarray}\label{est-2}\nonumber
&&\int_{Q_{N}\cap E}\int_{\{\xi: y+\xi\in (E\cap(Nj+ Q_N))\setminus Q_{(1+\delta)N}\}}
a(\xi)(w(y+\xi)-w(y))^2d\xi\,dy\\
&\le&{C\over \delta^{2+d+\kappa} N^\kappa} \int_{Q_N\cap E}\int_{\{\xi: y+\xi\in E\cap Q_N\}}
a(\xi)(w(y+\xi)-w(y))^2d\xi\,dy+  C{N^{d-\kappa}\over \delta^{2+d+\kappa}}.
\end{eqnarray}
To complete the estimate on the cubes neighbouring $Q_N$ it remains to estimate the interactions
between pairs of points `close to the boundary' of $Q_N$, where $w(y)=\langle z,y\rangle$. For such interactions we have
\begin{eqnarray}\label{est-3}\nonumber
&&\hskip-1cm\int_{(Q_N\setminus Q_{(1-\delta)N})\cap E}\int_{\{\xi: y+\xi\in E\cap(Q_{(1+\delta)N}\setminus Q_N)\}}
a(\xi)(w(y+\xi)-w(y))^2d\xi\,dy\\ \nonumber
&\le&\int_{(Q_N\setminus Q_{(1-\delta)N})\cap E}\int_{\{\xi: y+\xi\in E\cap(Q_{(1+\delta)N}\setminus Q_N)\}}
a(\xi)|z|^2|\xi|^2d\xi\,dy\ \\
&\le&C N^{d-1}\delta|z|^2.
\end{eqnarray}

It remains to estimate the interactions between points in $Q_N$ and in `non-neighbouring' cubes $jN+Q_N$ with $\|j\|_\infty\ge2$.
Using the decay estimate \eqref{ass-a} on $a$ we obtain
\begin{eqnarray}\nonumber
&&\sum_{j\in\ZZ^d\  \|j\|_\infty\ge1}\int_{Q_N\cap E}\int_{\{\xi: y+\xi\in E\cap(jN+ Q_N)\}}
a(\xi)(w(y+\xi)-w(y))^2d\xi\,dy\\
&\le&{1\over N^{2+d+\kappa}}\sum_{j\in\ZZ^d, \|j\|_\infty\ge2}{1\over |j|^{2+d+\kappa}}
\int_{Q_N\cap E}\int_{\{\xi: y+\xi\in E\cap(jN+ Q_N)\}}
(w(y+\xi)-w(y))^2d\xi\,dy.
\end{eqnarray}
Proceeding as in \eqref{est-inter} we then get
\begin{eqnarray}\label{est-4}\nonumber
&&\sum_{j\in\ZZ^d\  \|j\|_\infty\ge1}\int_{Q_N\cap E}\int_{\{\xi: y+\xi\in E\cap(jN+ Q_N)\}}
a(\xi)(w(y+\xi)-w(y))^2d\xi\,dy\\
&\le&C{1\over N^{\kappa}}\sum_{j\in\ZZ^d\setminus\{0\}}{1\over |j|^{2+d+\kappa}}
\int_{Q_N\cap E}\int_{\{\xi: y+\xi\in E\cap Q_N\}}a(\xi)
(w(y+\xi)-w(y))^2d\xi\,dy.
\end{eqnarray}

Taking into account estimates \eqref{est-1}--\eqref{est-4}
we finally get
\begin{eqnarray}\nonumber
&&\int_{Q_N\cap E}\int_{\{\xi: y+\xi\in E\setminus Q_N\}}
a(\xi)(w(y+\xi)-w(y))^2d\xi\,dy\\ \nonumber
&\le&C\Bigl({1\over \delta^{2+d+\kappa} N^\kappa}
+{1\over N^{\kappa}}\Bigr)
\int_{Q_N\cap E}\int_{\{\xi: y+\xi\in E\cap Q_N\}}
a(\xi)(w(y+\xi)-w(y))^2d\xi\,dy\\
&&+C N^{d-1}\delta|z|^2+ C{N^{d-\kappa}\over \delta^{2+d+\kappa}}.
\end{eqnarray}
which yields \eqref{cubeN-est}.
\end{proof}

\begin{proposition}[asymptotic homogenization formula]\label{def-fhom}
There exists the limit
\begin{eqnarray}\label{fhom-form}
&&f_{\rm hom}(z)=\nonumber
\lim_{N\in\NN,\ N\to+\infty}{1\over N^d}\min\Bigl\{\int_{Q_N\cap E}\int_{\{\xi: y+\xi\in Q_N\cap E\}}
a(\xi)(w(y+\xi)-w(y))^2d\xi\,dy: \\
&& \hskip3cm w(y)=\langle z,y\rangle\hbox{ on }Q_N\setminus Q_{(1-\delta)N} \hbox{ for some }\delta>0\Bigr\}\,.
\end{eqnarray}
\end{proposition}

\begin{proof}
With fixed $N\in\NN$ and $w$ as in \eqref{fhom-form} we extend $w$ to a $N$-periodic function with a slight abuse of notation.
Moreover, if $S\ge N$ we define
$$
w_S(y)=\begin{cases}
w(y) &\hbox{ if } y\in Q_{\lfloor {S\over N}\rfloor N}\\
\langle z,y\rangle &\hbox{ otherwise in } Q_S.
\end{cases}
$$

We can write
\begin{eqnarray}\label{estimone}
&&\nonumber
\int_{Q_S\cap E}\int_{\{\xi: y+\xi\in Q_S\cap E\}}
a(\xi)(w_S(y+\xi)-w_S(y))^2d\xi\,dy\\ \nonumber
&=&\sum_{\|j\|_\infty\le \lfloor {S\over N}\rfloor}
\int_{(jN+Q_N)\cap E}\int_{\{\xi: y+\xi\in Q_S\cap E\}}
a(\xi)(w_S(y+\xi)-w_S(y))^2d\xi\,dy\\ \nonumber
&&
+
\int_{(Q_S\setminus Q_{\lfloor {S\over N}\rfloor N})\cap E}\int_{\{\xi: y+\xi\in Q_S\cap E\}}
a(\xi)(w_S(y+\xi)-w_S(y))^2d\xi\,dy
\\ \nonumber
&=&\sum_{\|j\|_\infty\le \lfloor {S\over N}\rfloor}
\int_{(jN+Q_N)\cap E}\int_{\{\xi: y+\xi\in Q_{\lfloor {S\over N}\rfloor N+\delta N}\cap E\}}
a(\xi)(w(y+\xi)-w(y))^2d\xi\,dy\\ \nonumber
&&
+\sum_{\|j\|_\infty\le \lfloor {S\over N}\rfloor}
\int_{(jN+Q_N)\cap E}\int_{\{\xi: y+\xi\in (Q_S\setminus Q_{\lfloor {S\over N}\rfloor N+\delta N})\cap E\}}
a(\xi)(w_S(y+\xi)-w_S(y))^2d\xi\,dy\\ \nonumber
&&
+
\int_{(Q_S\setminus Q_{\lfloor {S\over N}\rfloor N})\cap E}\int_{\{\xi: y+\xi\in (Q_S\setminus Q_{\lfloor {S\over N}\rfloor N-\delta N})\cap E\}}
a(\xi)(w_S(y+\xi)-w_S(y))^2d\xi\,dy\\ \nonumber
&&
+\sum_{\|j\|_\infty\le \lfloor {S\over N}\rfloor-1}
\int_{(Q_S\setminus Q_{\lfloor {S\over N}\rfloor N})\cap E}\int_{\{\xi: y+\xi\in (jN+Q_N)\cap E\}}
a(\xi)(w_S(y+\xi)-w_S(y))^2d\xi\,dy\\
&& \nonumber
+\sum_{\|j\|_\infty= \lfloor {S\over N}\rfloor}
\int_{(Q_S\setminus Q_{\lfloor {S\over N}\rfloor N})\cap E}\int_{\{\xi: y+\xi\in Q_{\lfloor {S\over N}\rfloor N-\delta N}\cap (jN+Q_N)\cap E\}}
a(\xi)(w_S(y+\xi)-w_S(y))^2d\xi\,dy\\
\end{eqnarray}

By Proposition \ref{small-ext}, and the periodicity of $w$ we have
\begin{eqnarray}\label{est-11}
&&\nonumber
\int_{(jN+Q_N)\cap E}\int_{\{\xi: y+\xi\in Q_{\lfloor {S\over N}\rfloor N+\delta N}\cap E\}}
a(\xi)(w(y+\xi)-w(y))^2d\xi\,dy\\
&\le&
\Bigl(1+C\Bigl({1\over \delta^{2+d+\kappa} N^\kappa}
+ N^{d-1}\delta|z|^2\Bigr)\Bigr)
\int_{Q_N\cap E}\int_{\{\xi: y+\xi\in Q_N\cap E\}}
a(\xi)(w(y+\xi)-w(y))^2d\xi\,dy.\hskip1cm
\end{eqnarray}

Furthermore,
\begin{eqnarray}\label{est-12}
&&\nonumber
\int_{(Q_S\setminus Q_{\lfloor {S\over N}\rfloor N})\cap E}\int_{\{\xi: y+\xi\in (Q_S\setminus Q_{\lfloor {S\over N}\rfloor N-\delta N})\cap E\}}
a(\xi)(w_S(y+\xi)-w_S(y))^2d\xi\,dy\\
&=&\nonumber
\int_{(Q_S\setminus Q_{\lfloor {S\over N}\rfloor N})\cap E}\int_{\{\xi: y+\xi\in (Q_S\setminus Q_{\lfloor {S\over N}\rfloor N-\delta N})\cap E\}}
a(\xi)(\langle z,\xi\rangle)^2d\xi\,dy
\\
&\le &\nonumber
|Q_S\setminus Q_{\lfloor {S\over N}\rfloor N-\delta N}|  |z|^2\int_{\rr^d} a(\xi)|\xi|^2 d\xi
\\
&\le &
C S^{d-1}(N+\delta)|z|^2
\end{eqnarray}

Again, using the periodicity of $w(y)-\langle z,y\rangle$ and the decay assumption \eqref{ass-a} on $a$, we can estimate
\begin{eqnarray}\label{est-13}
&&\nonumber
\sum_{\|j\|_\infty\le \lfloor {S\over N}\rfloor-1}
\int_{(Q_S\setminus Q_{\lfloor {S\over N}\rfloor N})\cap E}\int_{\{\xi: y+\xi\in (jN+Q_N)\cap E\}}
a(\xi)(w_S(y+\xi)-w_S(y))^2d\xi\,dy\\
&\le&\nonumber
\sum_{\|j\|_\infty\le \lfloor {S\over N}\rfloor-1}\sum_{\|k\|_\infty= \lfloor {S\over N}\rfloor+1}
\int_{(kN+Q_N)\cap E}\int_{\{\xi: y+\xi\in (jN+Q_N)\cap E\}}
a(\xi)(w(y+\xi)-\langle z,y\rangle)^2d\xi\,dy\\
&\le&\nonumber
\sum_{\|j\|_\infty\le \lfloor {S\over N}\rfloor-1}\sum_{\|k\|_\infty= \lfloor {S\over N}\rfloor+1}
\int_{(kN+Q_N)\cap E}\int_{\{\xi: y+\xi\in (jN+Q_N)\cap E\}}
a(\xi)(w(y+\xi)-\langle z,y+\xi\rangle+\langle z,\xi\rangle)^2d\xi\,dy
\\
&\le&\nonumber 2
\sum_{\|j\|_\infty\le \lfloor {S\over N}\rfloor-1}\sum_{\|k\|_\infty= \lfloor {S\over N}\rfloor+1}
\int_{(kN+Q_N)\cap E}\int_{\{\xi: y+\xi\in (jN+Q_N)\cap E\}}
a(\xi)((w(y+\xi)-\langle z,y+\xi\rangle)^2+|z|^2|\xi|^2)d\xi\,dy
\\
&\le&\nonumber 2
\sum_{\|j\|_\infty\le \lfloor {S\over N}\rfloor-1}\sum_{\|k\|_\infty= \lfloor {S\over N}\rfloor+1}
\int_{Q_N\cap E}\int_{\{\xi: y+\xi\in Q_N\cap E\}}
a(\xi+(k-j)N)((w(y)-\langle z,y\rangle)^2+|z|^2|\xi+(k-j)N|^2)d\xi\,dy\\
&\le&
C\Bigl(\lfloor {S\over N}\rfloor+1\Bigr)^{d-1}
\Bigl(\int_{Q_N\cap E}|(w(y)-\langle z,y\rangle|^2+|z|^2N^d\Bigr)
\end{eqnarray}

As for the last term in \eqref{estimone}, we can use the same argument, noting that we have
$|\xi|\ge \delta N$,
which, using the decay estimate \eqref{ass-a} on $a$, gives
\begin{eqnarray}\label{est-14}\nonumber
&&\sum_{\|j\|_\infty= \lfloor {S\over N}\rfloor}
\int_{(Q_S\setminus Q_{\lfloor {S\over N}\rfloor N})\cap E}\int_{\{\xi: y+\xi\in Q_{\lfloor {S\over N}\rfloor N-\delta N}\cap (jN+Q_N)\cap E\}}
a(\xi)(w_S(y+\xi)-w_S(y))^2d\xi\,dy\\
&\le&
C\Bigl(\lfloor {S\over N}\rfloor+1\Bigr)^{d-1}{1\over \delta^{2+d+\kappa} N^{2+d+\kappa}}
\Bigl(\int_{Q_N\cap E}|(w(y)-\langle z,y\rangle|^2+|z|^2N^d\Bigr)
\end{eqnarray}

If for $t>0$ and $u\in L^2(Q_t)$ we set
$$
g_t(u)={1\over t^d}\int_{Q_t\cap E}\int_{\{\xi: y+\xi\in Q_t\cap E\}}
a(\xi)(u(y+\xi)-u(y))^2d\xi\,dy,
$$
then gathering estimates \eqref{est-11}--\eqref{est-14} we have
\begin{eqnarray}\label{est-15}\nonumber
g_S(w_S)&\le& {N^d\over S^d}\Bigl\lfloor {S\over N}\Bigr\rfloor^d \Bigl(1+{C\over \delta^{2+d+\kappa} N^\kappa} \Bigr)
g_N(w)+ {1\over N}\delta|z|^2+C {1\over S}(N+\delta)|z|^2\\
&&+\nonumber
C{1\over S^d}\Bigl(\Big\lfloor {S\over N}\Big\rfloor+1\Bigr)^{d-1}\Bigl(1+{1\over \delta^{2+d+\kappa} N^{2+d+\kappa}}\Bigr)
\Bigl(\int_{Q_N\cap E}|(w(y)-\langle z,y\rangle|^2+|z|^2N^d\Bigr).\\
\end{eqnarray}
If we set
\begin{eqnarray}f_S(z)&=&\nonumber
{1\over S^d}\min\Bigl\{\int_{Q_S\cap E}\int_{\{\xi: y+\xi\in Q_S\cap E\}}
a(\xi)(w(y+\xi)-w(y))^2d\xi\,dy: \\
&& \hskip2cm w(y)=\langle z,y\rangle\hbox{ on }Q_S\setminus Q_{(1-\delta)S} \hbox{ for some }\delta>0\Bigr\}\,,
\end{eqnarray}
this implies that
\begin{eqnarray}\label{est-16}\nonumber
\limsup_{S\to+\infty}f_S(z)\le
\Bigl(1+{C\over \delta^{2+d+\kappa} N^\kappa} \Bigr)
g_N(w)+ {1\over N}\delta|z|^2.
\end{eqnarray}
By the arbitrariness of $w$ and $\delta$ we then obtain
$\limsup\limits_{S\to+\infty}f_S(z)\le\liminf\limits_{N\to+\infty} f_N(z)$,
which proves the claim.
\end{proof}

\begin{remark}\label{quaform}\rm
Note that $f_{\rm hom}$ is a quadratic form and hence there exists a symmetric matrix $A_{\rm hom}$ such that
\begin{equation}
f_{\rm hom}(z)= \langle A_{\rm hom}z,z\rangle.
\end{equation}
Indeed, let $w_1$ and $w_2$ be test functions in \eqref{fhom-form}
for $z_1$ an $z_2$ then using $w_1+w_2 $ and $w_1-w_2 $ as test functions we get the inequality
$$
2\Bigl(f_{\rm hom}(z_1)+f_{\rm hom}(z_2)\Bigr)\ge f_{\rm hom}(z_1+z_2)+f_{\rm hom}(z_1-z_2).
$$
A symmetric argument shows that the converse inequality, and then the equality, holds, giving the claim
(see \cite{BDF} Remark 7.12).
\end{remark}

\begin{remark}\label{rem-fhom}\rm
Thanks to Proposition \ref{small-ext} $f_{\rm hom}$ can be equivalently written
\begin{eqnarray}\label{fhom-form-2}
f_{\rm hom}(z)&=&\nonumber
\lim_{T\in\NN,\ T\to+\infty}{1\over T^d}\min\Bigl\{\int_{Q_T\cap E}\int_{\{\xi: y+\xi\in \rr^d\cap E\}}
a(\xi)(w(y+\xi)-w(y))^2d\xi\,dy: \\
&& w(y)-\langle z,y\rangle\  T\hbox{-periodic}, \
w(y)=\langle z,y\rangle\hbox{ on }Q_T\setminus Q_{(1-\delta)T} \hbox{ for some }\delta>0\Bigr\}\,,
\end{eqnarray}
where the second integral is extended to the whole $\rr^d\cap E$.
\end{remark}

\begin{proposition}\label{boundper}
Let  $f_0(z)$ be given by \eqref{cpformula}. Then $f_{\rm hom}\ge f_0$.
\end{proposition}

\begin{proof}
By Proposition \ref{def-fhom} we have
\begin{eqnarray}\label{asy-phom}
f_{\rm hom}(z)&\ge&\nonumber
\limsup_{N\in\NN,\ N\to+\infty}{1\over N^d}\min\Bigl\{\int_{Q_N\cap E}\int_{\{\xi: y+\xi\in E\}}
a(\xi)(w(y+\xi)-w(y))^2d\xi\,dy: \\
&& \hskip2cm w(y)-\langle z,y\rangle  \ N\hbox{-periodic}\Bigr\}\,.
\end{eqnarray}

We now show that the right-hand side of \eqref{asy-phom} is equal to $f_0(z)$.
It suffices to show one inequality, the other being trivial since $1$-periodic functions are also $N$-periodic. To that end, given $w$ test function for the problem in \eqref{asy-phom} it suffices to construct the $1$-periodic function
$$
\overline w(y)={1\over N^d}\sum_{k\in\{1,\ldots, N\}^d}w(y+k),
$$
and note that, by periodicity and convexity,
\begin{eqnarray}\label{c-phom}
&&\nonumber
\int_{Q_1\cap E}\int_{\{\xi: y+\xi\in E\}}
a(\xi)(\overline w(y+\xi)-\overline w(y))^2d\xi\,dy\\
&=& \nonumber
{1\over N^d}
\int_{Q_N\cap E}\int_{\{\xi: y+\xi\in E\}}
a(\xi)(\overline w(y+\xi)-\overline w(y))^2d\xi\,dy\\
&\le&\nonumber
{1\over N^d}{1\over N^d}\sum_{k\in\{1,\ldots, N\}^d}
\int_{Q_N\cap E}\int_{\{\xi: y+\xi\in E\}}
a(\xi)(w(y+k+\xi)-w(y+k))^2d\xi\,dy
\\
&=&\nonumber
{1\over N^d}{1\over N^d}\sum_{k\in\{1,\ldots, N\}^d}
\int_{Q_N\cap E}\int_{\{\xi: y+\xi\in E\}}
a(\xi)(w(y+\xi)-w(y))^2d\xi\,dy
\\
&=&
{1\over N^d}
\int_{Q_N\cap E}\int_{\{\xi: y+\xi\in E\}}
a(\xi)(w(y+\xi)-w(y))^2d\xi\,dy.
\end{eqnarray}
This proves the claim.
\end{proof}

\begin{proposition}[upper bound]\label{ubound}
Let  $f_0$ be given by \eqref{cpformula}. Then for all $u\in H^1(\Omega)$ there exists a sequence $u_\e\to u$ such that
\begin{equation}\label{ubeq}
\limsup_{\e\to 0}F_\e(u_\e)\le \int_\Omega f_0(\nabla u)\dx.
\end{equation}
\end{proposition}

\begin{proof}
By a diagonalization  argument it suffices to show that for all $\eta>0$ there exists $u_\e\to u$ such
that
$$
\limsup_{\e\to 0}F_\e(u_\e)\le \int_\Omega f_0(\nabla u)\dx+ \eta.
$$

Note that the same argument as in Remark \ref{quaform} shows that $f_0$ is a quadratic form, so that the
right-hand side of \eqref{ubeq} defines a strongly continuous functional in $H^1(\Omega)$.
Then, by a standard density argument (see \cite{GCB} Remark 1.29) it suffices to consider the case when $u$ is piecewise affine.
We only consider the case when  $\Omega$ is bounded and the gradient of $u$ takes two values, since the general case only results in a heavier notation.

Upon a translation and reflection argument, we may assume  that
$$
u(x)=\min\{ \langle z_1,x\rangle,  \langle z_2,x\rangle\}.
$$
For $i\in\{1,2\}$, we then set
\begin{eqnarray*}
S=\{x\in\Omega: \langle z_1-z_2,x\rangle=0\},
\qquad
\Omega_i=\{x\in\Omega: u(x)=\langle z_i,x\rangle\}\setminus S.
\end{eqnarray*}

We may suppose that $w_i$ are functions such that $w_i(y)-\langle z_i,x\rangle$ is $1$-periodic and
$$
\int_{Q_1\cap E}\int_{E} a(x-y)(w_i(x)-w_i(y))^2 dx\,dy= f_0(z_i).
$$
Then, the functions
$$
u^i_\e(x)= \e\, w_i\Bigl({x\over\e}\Bigr)
$$
converge to $\langle z_i,x\rangle$ in $\Omega_i$ as $\e\to0$, and, if  we set
\begin{eqnarray}\label{sep-def}
F^i_\e(u^i_\e):=\int_{\Omega_i\cap\e E}\int_{{1\over\e}((\Omega_i\cap\e E)-x)} a(\xi)
\Bigl({u^i_\e(x+\e\xi)-u^i_\e(x)\over\e}\Bigr)^2d\xi \dx,
\end{eqnarray}
then we have
\begin{eqnarray}\label{sep-est}
\limsup_{\e\to0}F^i_\e(u^i_\e)\le |\Omega_i|f_0(z_i).
\end{eqnarray}

To check this, for all $\e>0$ define the sets of indices
\begin{eqnarray*}
I^\e_i&=&\{k\in  \ZZ^d: \e (k+Q)\cap \Omega_i\neq\emptyset\},
\end{eqnarray*}
for $i\in\{1,2\}$. Then
\begin{eqnarray*}
F^i_\e(u^i_\e)&\le&\sum_{k\in I^\e_i}\int_{\e (k+Q)\cap\e E}\int_{{1\over\e}(\Omega_i\cap\e E-x)} a(\xi)
\Bigl({u^i_\e(x+\e\xi)-u^i_\e(x)\over\e}\Bigr)^2d\xi \dx\\
&=&\sum_{k\in I^\e_i}\e^d\int_{ (k+Q)\cap E}\int_{{1\over\e}\Omega_i\cap E-y} a(\xi)
(w_i(y+\xi)-w_i(y))^2d\xi \,dy\\
&\le&\sum_{k\in I^\e_i}\e^d\int_{ (k+Q)\cap E}\int_{E-y} a(\xi)
(w_i(y+\xi)-w_i(y))^2d\xi \,dy
\\
&=&\sum_{k\in I^\e_i}\e^d\int_{Q\cap E}\int_{E} a(x-y)
(w_i(x)-w_i(y))^2d\xi \,dy
\\
&=&\# I^\e_i \e^d f_0(z_i)= |\Omega_i|f_0(z_i) +o(1)
\end{eqnarray*}
as $\e\to 0$.

Given $\delta>0$ we may modify each sequence $\{u^i_\e\}$ to a new sequence
$\{u^{\delta,i}_\e\}$ such that
\begin{eqnarray*}
 u^{\delta,i}_\e(x)= \langle z_i,x\rangle \hbox{ if dist}(x,S)<\delta, x\in \Omega_i,&&
 u^{\delta,i}_\e(x)= u^{i}_\e(x) \hbox{ if dist}(x,S)\ge 2\delta, x\in \Omega_i,
\\
 \int_{\Omega_1}|u^{\delta,1}_\e(x)|^2dx+\int_{\Omega_2}|u^{\delta,2}_\e(x)|^2dx\le  C,&\hbox{ and }&
 \limsup_{\e\to 0}(F^i_\e(u^{\delta,i}_\e)-
F^i_\e(u^i_\e))= o(1)
\end{eqnarray*}
as $\delta\to0$. This can be done using the construction of Proposition \ref{bvp},
with $A$ the intersection of a large ball containing $\Omega$ with each of the two half spaces with $S$ as boundary.

We then define the sequence
$$
u_\e(x)=u^{\delta,i}_\e \hbox{ if } x\in \Omega_i, \ i\in\{1,2\}.
$$
The only thing to check now is that
\begin{eqnarray}\label{sep-def}
\limsup_{\e\to0}\int_{\Omega_1\cap\e E}\int_{{1\over\e}(\Omega_2\cap\e E-x)} a(\xi)
\Bigl({u^{2,\delta}_\e(x+\e\xi)-u^{1,\delta}_\e(x)\over\e}\Bigr)^2d\xi \dx= o(1)
\end{eqnarray}
as $\delta\to0$. To that end, first set
$$
\Omega_i^\delta=\{x\in\Omega_i: \hbox{dist}(x,S)>\delta\},
$$
and estimate
\begin{eqnarray*}
&&\int_{\Omega_1^\delta\cap\e E}\int_{{1\over\e}(\Omega_2\cap\e E-x)} a(\xi)
\Bigl({u^{2,\delta}_\e(x+\e\xi)-u^{1,\delta}_\e(x)\over\e}\Bigr)^2d\xi \dx\\
&\le&\int_{\Omega_1^\delta\cap\e E}\int_{{1\over\e}(\Omega_2\cap\e E-x)} {1\over\e^2} a(\xi)
(|u^{2,\delta}_\e(x+\e\xi)|^2+|u^{1,\delta}_\e(x)|^2)d\xi \dx
\\
&\le&\Bigl(\int_{\Omega_1}|u^{1,\delta}_\e(x)|^2\dx+\int_{\Omega_2}|u^{2,\delta}_\e(x)|^2\dx\Bigr)
 {1\over\e^2}\int_{\{|\xi|>\delta/\e\}} a(\xi)d\xi
\\
&\le&C{\e^\kappa\over \delta^{2+\kappa}}.
\end{eqnarray*}
In the same way we estimate the term
\begin{eqnarray*}
&&\int_{(\Omega_1\setminus \Omega_1^\delta)\cap\e E}\int_{{1\over\e}(\Omega_2^\delta\cap\e E-x)} a(\xi)
\Bigl({u^{2,\delta}_\e(x+\e\xi)-u^{1,\delta}_\e(x)\over\e}\Bigr)^2d\xi \dx\\
&\le&\int_{\Omega_1\cap\e E}\int_{{1\over\e}(\Omega_2^\delta\cap\e E-x)} a(\xi)
\Bigl({u^{2,\delta}_\e(x+\e\xi)-u^{1,\delta}_\e(x)\over\e}\Bigr)^2d\xi \dx\\
&\le&C{\e^\kappa\over \delta^{2+\kappa}}.
\end{eqnarray*}
Finally, using the Lipschitz continuity of $u$, we have
\begin{eqnarray*}
&&\int_{(\Omega_1\setminus \Omega_1^\delta)\cap\e E}\int_{{1\over\e}((\Omega_2\setminus \Omega_2^\delta)\cap\e E-x)} a(\xi)
\Bigl({u^{2,\delta}_\e(x+\e\xi)-u^{1,\delta}_\e(x)\over\e}\Bigr)^2d\xi \dx\\
&=&\int_{(\Omega_1\setminus \Omega_1^\delta)\cap\e E}\int_{{1\over\e}((\Omega_2\setminus \Omega_2^\delta)\cap\e E-x)} a(\xi)
\Bigl({u(x+\e\xi)-u(x)\over\e}\Bigr)^2d\xi \dx\\
&\le&C\int_{(\Omega_1\setminus \Omega_1^\delta)}\int_{{1\over\e}((\Omega_2\setminus \Omega_2^\delta)-x)} a(\xi)
|\xi|^2d\xi \dx\\
&\le&C\delta\int_{\rr^d} a(\xi)
|\xi|^2d\xi \dx= C\delta
\end{eqnarray*}
Gathering the previous estimates and letting $\e\to0$ we obtain
\begin{eqnarray*}
\limsup_{\e\to0}\int_{\Omega_1\cap\e E}\int_{{1\over\e}(\Omega_2\cap\e E-x)} a(\xi)
\Bigl({u^{2,\delta}_\e(x+\e\xi)-u^{1,\delta}_\e(x)\over\e}\Bigr)^2d\xi \dx\le C\delta,
\end{eqnarray*}
which proves \eqref{sep-def}.
\end{proof}

%

\subsection*{Acknowledgments.}
The authors acknowledge the MIUR Excellence Department Project awarded to the Department of Mathematics, University of Rome Tor Vergata, CUP E83C18000100006.

\end{document}